\documentclass[12pt]{amsart}
\usepackage[centering]{geometry}

\usepackage{amsmath,amsthm,amssymb}
\usepackage{times}
\usepackage{enumerate}

\theoremstyle{plain}
\newtheorem*{theorem*}{Theorem}
\newtheorem{theorem}{Theorem}[section]

\newtheorem{lemma}[theorem]{Lemma}

\newcommand{\be}{\begin{equation}}
\newcommand{\ee}{\end{equation}}
\newcommand{\parallelsum}{\mathbin{\!/\mkern-5mu/\!}}
\newcommand{\lt}{\left}
\newcommand{\rt}{\right}
\newcommand{\goto}{\rightarrow}
\newcommand{\R}{\mathbb{R}}
\newcommand{\e}{\epsilon}
\newcommand{\hZ}{\hat{Z}}
\newcommand{\hf}{\hat{f}}
\newcommand{\hF}{\hat{F}}
\newcommand{\uen}{u^{\e, n}}
\newcommand{\us}{u^{\e, n, \sigma}}
\newcommand{\Gs}{G_{\e, n, \sigma}}
\newcommand{\vp}{\varphi}
\newcommand{\uk}{u^{\e_k, n_k}}
\newcommand{\ud}{\underline{d}}

\newcommand{\txn}{\partial_{x_0}^n}
\newcommand{\tyn}{\partial_{y_0}^n}
\newcommand{\nue}{\nu_{x^\e}}

\theoremstyle{definition}
\newtheorem{defin}[theorem]{Definition}
\newtheorem{remark}[theorem]{Remark}

\numberwithin{equation}{section}
\numberwithin{equation}{section}

\begin{document}
\setlength{\baselineskip}{1.2\baselineskip}

\title[Motion of level set by general curvature]
{Motion of level set by general curvature}

\author{Ling Xiao}
\address{Department of Mathematics, Rutgers University,
Piscataway, NJ 08854}
\email{lx70@math.rutgers.edu}

\begin{abstract}
In this paper, we study the motion of level sets by general curvature. The difficulty in this setting is for a general curvature function, it's only well defined in an admissible cone. In order to extend the existence result to outside the cone we introduce a new approximation function $\hf^n$ (see \eqref{ew1}). Moreover, using the idea in \cite{CS}, we give an elliptic approach for the Ben-Andrews' non-collapsing result in fully nonlinear curvature flows. We hope this approach can be generalized to a wider class of elliptic equations.

\end{abstract}

\maketitle

\section{Introduction}
\label{int}
\setcounter{equation}{0}
In this paper, we are going to study a level set approach for general curvature flow. More specifically,
given an initial hypersurface $\Gamma_0,$ select some continuous function $g:\R^n\goto\R$ so that
\be\label{1.1}
\Gamma_0=\{x\in\R^n|g(x)=0\}.
\ee
We are going to consider the parabolic PDE
\be\label{1.2}
\begin{aligned}
&u_t=|\nabla u|F\lt(A[\Gamma_t]\rt)=F\lt(\gamma^{ik}u_{kl}\gamma^{lj}\rt)\\
&u=g\,\,\mbox{on $\R^n\times\{t=0\}$},
\end{aligned}
\ee
where $\gamma^{ij}=\delta_{ij}-\frac{u_iu_j}{|Du|^2},$ $\Gamma_t\equiv\{x\in\R^n|u(x, t)=0\},$ and $A[\Gamma_t]$ denotes the second fundamental form of the hypersurface $\Gamma_t.$ The PDE \eqref{1.2} says that each level set of $u$
evolves according to a general curvature function, at least at the region where $u$ is smooth and $|Du|\neq 0.$
One can derive that when $\Gamma_t$ evolves according to its mean curvature then
\[F\lt(\gamma^{ik}u_{kl}\gamma^{lj}\rt)=(\delta_{ij}-u_iu_j/|Du|^2)u_{ij}.\]

Evans and Spruck in their series papers (see \cite{ES1, ES2, ES3, ES4}) study existence, uniqueness, and regularities of solutions of level set mean curvature flow. They also study the relationship between level set mean curvature flow and classical mean curvature flow. Around the same time, Chen, Giga and Goto study the existence and uniqueness of generalized level set mean curvature flow (see \cite{CGG}). Following these ideas, generalized motion of noncompact hypersurfaces
with normal velocity depending on the normal direction and the curvature
tensor have also been studied by Ishii and Souganidis, and Goto (see \cite{IS, G}). In their works, they study equation
\be\label{1.3}
u_t-F(Du, D^2u)=0,
\ee
where function $F$ is globally elliptic, i.e. $F(p, X)\geq F(p, Y)$ if $X\geq Y.$ Under the global ellipticity assumption, in \cite{CMP} Chambolle, Morini, and Ponsiglione  study general level set nonlocal curvature motions. However, the global ellipticity condition isn't satisfied when
$\Gamma_t$ moves by $k$-th mean curvature or curvature quotient.

In this paper, we are going to prove the existence and uniqueness result for motion by general curvature which is not globally elliptic.
Our main results are the following:

\begin{theorem}\label{thint1}
Assume $g:\R^n\goto\R$ is continuous and satisfying that
\[\mbox{$g$ is constant on $\R^n\cap\{|x|\geq S\}.$}\]
Then, there exists a weak solution $u$ of \eqref{ws1}, such that
\[\mbox{$u$ is constant on $\R^n\times[0, \infty)\cap\{|x|+t\geq R\}.$}\]
\end{theorem}

Moreover, we give a new approach to Ben-Andrews' non-collapsing results (see \cite{A, ALM}). Our approach is an elliptic approach based on the idea of \cite{CS}.
Combining Ben-Andrews' non-collapsing results with our result in Section \ref{secci}, we are able to extend the non-collapsing results to the weak flow:
\begin{theorem}\label{thew2}
Let $\Gamma_0$ be a hypersurface satisfies the following conditions: $\Gamma_0$ can be approached by a sequence of smooth hypersurfaces, which have positive general curvature and satisfy $\alpha$-Andrews condition. If $\{\Gamma_t\}$ is a compact level set general curvature flow with initial hypersurface $\Gamma_0,$ then $\{\Gamma_t\}$ is an $\alpha$-Andrews level set flow.
\end{theorem}

\section{Definition and properties of weak solutions}
\label{secws}
\subsection {Weak solutions}\label{subws}
\setcounter{equation}{0}
We study the following equation:
\be\label{ws1}
\begin{aligned}
&u_t=|\nabla u|F\lt(A[\Gamma_t]\rt)=F\lt(\gamma^{ik}u_{kl}\gamma^{lj}\rt)\\
&u=g\,\,\mbox{on $\R^n\times\{t=0\}$},
\end{aligned}
\ee
where $\gamma^{ij}=\delta_{ij}-\frac{u_iu_j}{|Du|^2}.$
Now, let $\kappa[\gamma^{ik}u_{kl}\gamma^{lj}]=(\kappa_1, \cdots, \kappa_n)$ be the eigenvalues of
$\{\gamma^{ik}u_{kl}\gamma^{lj}\},$ and let $f(\kappa)=F\lt(\gamma^{ik}u_{kl}\gamma^{lj}\rt).$
We assume the function $f$ satisfies the following fundamental structure conditions:\\
\be\label{ws6}
f_i(\kappa)\equiv\frac{\partial f(\kappa)}{\partial\kappa_i}>0\,\,\mbox{in $K$, $1\leq i\leq n$,}
\ee
\be\label{ws7}
\mbox{$f$ is a concave function in $K$,}
\ee
\be\label{ws8}
\mbox{$f>0$ in K, $f=0$ on $\partial K$,}
\ee
where $K\subset\R^n$ is an open symmetric convex cone such that
\be\label{ws9}
\bar{K}^+_n:=\{\kappa\in\R^n:\,\,\mbox{each component $\kappa_i\geq 0$}\}\varsubsetneq\bar{K}.
\ee
In addition, we shall assume that $f$ is normalized
\be\label{ws10}
f(1, \cdots, 1)=1,
\ee
satisfies the more technical assumptions
\be\label{ws11}
\mbox{$f$ is homogeneous of degree one.}
\ee

\begin{defin}\label{wsde1}
We say a function $\phi\in C^\infty(\R^{n+1})$ is $\textit{admissible}$ on a domain $D$ if and only if
for any $x\in D$ we have,
\be\label{ws2}
\left\{
\begin{aligned}
&F\lt(\gamma^{ik}\phi_{kl}\gamma^{lj}\rt)> 0\\
&\mbox{if $D\phi\neq 0$},\\
\end{aligned}
\right.
\ee
and
\be\label{ws3}
\left\{
\begin{aligned}
&F\lt(\gamma^{ik}_\eta\phi_{kl}\gamma^{lj}_\eta\rt)>0\\
&\mbox{for some $\eta\in\R^n$ with $|\eta|\leq 1,$ if $D\phi= 0$, },\\
\end{aligned}
\right.
\ee
where $\gamma^{ij}_{\eta}=\delta_{ij}-\eta_i\eta_j.$
\end{defin}

We will use  $\Lambda(D)$ to denote the class of $\textit{admissble}$ function in domain D.

\begin{defin}\label{wsde2}
A function $u\in C(\R^n\times[0, \infty) \cap L^{\infty}(\R^n\times[0, \infty))$ is a weak subsolution (supersolution) of
\eqref{ws1} provided that for any $\phi\in C^{\infty}(\R^{n+1})\cap\Lambda(D),$ where $D$ is an open set in $\R^n\times[0, \infty),$ if
$u-\phi$ has a local maximum (minimum) at a point $(x_0, t_0)\in D,$ then
at this point
\be\label{ws4}
\left\{
\begin{aligned}
&\phi_t\leq (\geq) F\lt(\gamma^{ik}\phi_{kl}\gamma^{lj}\rt)\\
&\mbox{if $D\phi\neq 0$},\\
\end{aligned}
\right.
\ee
and
\be\label{ws5}
\left\{
\begin{aligned}
&\phi_t\leq (\geq) F\lt(\gamma^{ik}_\eta\phi_{kl}\gamma^{lj}_\eta\rt)\\
&\mbox{for some $\eta\in\R^n$ with $|\eta|\leq 1,$ if $D\phi= 0$}.\\
\end{aligned}
\right.
\ee
\end{defin}
It will be convenient to have at hand an alternative definition. We write
$z=(x, t),$ $z_0=(x_0, t_0),$ and below implicitly sum $i, j$ from 1 to n.
\begin{defin}\label{wsde3}
A  function $u\in C(\R^n\times [0, \infty))\cap L^{\infty}(\R^n\times[0, \infty))$
is a weak subsolution (supersolution) of \eqref{ws1} if whenever $(x_0, t_0)\in\R^n\times(0, \infty)$ and
\be\label{ws13}
u(x, t)\leq (\geq) u(x_0, t_0)+p\cdot(x-x_0)+q(t-t_0)+\frac{1}{2}(z-z_0)^\top R(z-z_0)
+o(|z-z_0|^2)\,\,\,\mbox{as $z\goto z_0$}
\ee
for some $p\in \R^n,$ $q\in \R,$ and $R=\{r_{ij}\}\in S^{(n+1)\times (n+1)}$
then
\be\label{ws14}
q\leq (\geq) F(\gamma^{ik}_pr_{kl}\gamma_p^{lj})\,\,\mbox{if $p\neq 0$ and $\kappa[\gamma^{ik}_pr_{kl}\gamma^{lj}_p]\in K,$}
\ee
and
\be\label{ws15}
q\leq (\geq) F(\gamma^{ik}_\eta r_{kl}\gamma_\eta^{lj})
\ee
for some $|\eta|\leq 1$ if $p=0$ and $\kappa[\gamma^{ik}_\eta r_{kl}\gamma_\eta^{lj}]\in K.$
\end{defin}

\begin{defin}\label{wsde4}
A function  $u\in C(\R^n\times[0, \infty)) \cap L^{\infty}(\R^n\times[0, \infty))$ is a weak solution of
\eqref{ws1} provided $u$ is both a weak subsolution and a weak supersolution.
\end{defin}

\subsection{Properties of weak solutions}
\label{subpws}
\begin{theorem}\label{thws1}
(i) Assume $u_k$ is a weak solution of \eqref{ws1} for $k=1, 2, \cdots$
and $u_k\goto u$ bounded and locally uniformly on $\R^n\times[0, \infty).$
Then $u$ is a weak solution.\\
(ii) An analogous assertion holds for weak subsolutions and supersolutions.
\end{theorem}
\begin{proof}
1. Choose $\phi\in C^{\infty}(\R^{n+1})\cap\Lambda(D)$ and suppose $u-\phi$ has a strict local
maximum at some point $(x_0, t_0)\in D\subset\R^n\times (0, \infty).$
As $u_k\goto u$ uniformly near $(x_0, t_0),$ $u_k-\phi$ has a local maximum at a point
$(x_k, t_k) (k=1, 2, \cdots)$ with $(x_k, t_k)\goto (x_0, t_0)$ as $k\goto\infty.$
Since $u_k$ is a weak solution, we have either
\be\label{ws16}
\phi_t\leq F(\gamma^{ik}\phi_{kl}\gamma^{lj})
\ee
at $(x_k, t_k)$ if $D\phi(x_k, t_k)\neq 0,$
or
\be\label{ws17}
\phi_t\leq F(\gamma^{ik}_{\eta^k}\phi_{kl}\gamma^{lj}_{\eta^k})
\ee
at $(x_k, t_k)$ for some $\eta^k\in \R^n$ with $|\eta^k|\leq 1$ if $D\phi(x_k, t_k)=0.$

2. Assume first $D\phi(x_0, t_0)\neq 0,$ then $D\phi(x_k, t_k)\neq 0$ for all large
enough $k.$ Hence we may pass to limits in the inequality \eqref{ws16} and get
\be\label{ws18}
\phi_t\leq F(\gamma^{ik}\phi_{kl}\gamma^{lj})\,\,\mbox{at $(x_0, t_0).$}
\ee

3. Next suppose $D\phi(x_0, t_0)=0.$ We set
\be\label{ws19}
\xi^k\equiv\left\{
\begin{aligned}
(D\phi/|D\phi|)(x_k, t_k)\,\,\mbox{if}\,\, D\phi(x_k, t_k)&\neq 0,\\
\eta^k\,\, \mbox{if}\,\, D\phi(x_k, t_k)&=0.
\end{aligned}
\right.
\ee
Passing if necessary to a subsequence we may assume $\xi^k\goto\eta,$ then
$|\eta|\leq 1.$ Utilizing \eqref{ws16} and \eqref{ws17} we deduce
\be\label{ws20}
\phi_t\leq F(\gamma^{ik}_{\eta}\phi_{kl}\gamma^{lj}_{\eta})\,\,\mbox{at $(x_0, t_0).$}
\ee

4. If $u-\phi$ has only a local maximum at $(x_0, t_0)$ we may apply the above
argument to
\be\label{ws21}
\psi(x, t)\equiv\phi(x, t)+|x-x_0|^4+(t-t_0)^4.
\ee
It's easy to see that when $\phi\in\Lambda(D)$ we have $\psi\in\Lambda(\tilde{D}),$ where $(x_0, t_0)\in \tilde{D}\subset D.$
Hence, we showed $u$ is a weak subsolution  of \eqref{ws1}. Similarly, we can show that $u$ is a weak supersolution. Therefore, $u$ is a weak solution  of \eqref{ws1}.
\end{proof}

\begin{theorem}
\label{thws2}
Assume u is a weak solution of \eqref{ws1} and $\Psi:\R\goto\R$
is a continuous monotone increasing function. Then $v\equiv\Psi(u)$ is a weak solution of \eqref{ws1}.
\end{theorem}

\begin{proof}
1. Let $\phi\in C^{\infty}(\R^{n+1})\cap\Lambda(D)$ and $\Psi$ is smooth, $\Psi'>0$ on $\R.$
Suppose $v-\phi$ has a local maximum at $(x_0, t_0)\in D.$ Adding as necessary a constant to $\phi,$
we may assume
\be
\label{ws22}
\left\{
\begin{aligned}
v(x_0, t_0)&=\phi(x_0, t_0),\\
v(x, t)&\leq\phi(x, t)\,\,\mbox{for all $(x, t)$ near $(x_0, t_0)$.}
\end{aligned}
\right.
\ee
Since $\Psi'>0$ in $\R,$ $\Phi\equiv\Psi^{-1}$ is defined and smooth near $u(x_0, t_0)$
with $\Phi'>0.$
Furthermore,
\be\label{ws23}
\left\{
\begin{aligned}
u(x_0, t_0)&=\psi(x_0, t_0)\\
u(x, t)&\leq\psi(x, t)\,\,\mbox{for $(x, t)$ near $(x_0, t_0),$}
\end{aligned}
\right.
\ee
where $\psi\equiv\Phi(\phi).$

2. Since $u$ is a weak solution and
\[\kappa[\gamma^{ik}_{\psi}\psi_{kl}\gamma^{lj}_{\psi}]
=\Phi'\kappa[\gamma^{ik}_{\phi}\phi_{kl}\gamma^{lj}_{\phi}]\subset K\,\,\mbox{when $D\psi(x_0, t_0)\neq 0,$}\]
\[\kappa[\gamma^{ik}_{\eta}\psi_{kl}\gamma^{lj}_{\eta}]
=\Phi'\kappa[\gamma^{ik}_{\eta}\phi_{kl}\gamma^{lj}_{\eta}]\subset K\,\,\mbox{when $D\psi(x_0, t_0)=0$},\]
we have
\be\label{ws24}
\psi_t\leq F(\gamma^{ik}_{\psi}\psi_{kl}\gamma^{lj}_{\psi})\,\,\mbox{at $(x_0, t_0)$ if $D\psi(x_0, t_0)\neq 0$}
\ee
and
\be\label{ws25}
\psi_t\leq F(\gamma^{ik}_{\eta}\psi_{kl}\gamma^{lj}_{\eta})\,\,\mbox{at $(x_0, t_0)$ if $D\psi(x_0, t_0)=0.$}
\ee
Consequently we have at $(x_0, t_0)$
\be\label{ws26}
\phi_t\leq F(\gamma^{ik}_{\phi}\phi_{kl}\gamma^{lj}_{\phi})\,\,\mbox{at $(x_0, t_0)$ if $D\phi(x_0, t_0)\neq 0$}
\ee
and
\be\label{ws27}
\phi_t\leq F(\gamma^{ik}_{\eta}\phi_{kl}\gamma^{lj}_{\eta})\,\,\mbox{at $(x_0, t_0)$ if $D\phi(x_0, t_0)=0.$}
\ee
Similarly we have the opposite inequalities to \eqref{ws26}, \eqref{ws27} when $v-\phi$
have a local minimum at $(x_0, t_0).$

3. We have so far shown that $v=\Psi(u)$ is a weak solution provided $\Psi$ is smooth with
$\Psi'>0.$ Approximating and using Theorem \ref{thws1} we draw the same conclusion if $\Psi'\geq 0$
on $\R.$

4. Finally suppose only that $\Psi$ is continuous and monotone increasing. We construct a sequence of
smooth functions $\{\Psi^k\}_{k=1}^\infty,$ $(\Psi^k)'\geq 0,$ so that $\Psi^k\goto\Psi$
uniformly on $[-\|u\|_{L^\infty}, \|u\|_{L^\infty}].$ Consequently we have
\[v^k=\Psi^k(u)\goto v=\Psi(u)\]
bounded and uniformly. Then Theorem \ref{thws1} asserts $v$ to be a weak solution.
\end{proof}

Following the proof in \cite{ES1}, we can prove the following comparison theorem and contraction property:
\begin{theorem}\label{thws3}
Assume that $u$ is a weak subsolution and $v$ is a weak supersolution of \eqref{ws1}.
Suppose further
\[u\leq v\,\,\mbox{on $\R^n\times\{t=0\}.$}\]
Finally assume
\be\label{ws30}
\mbox{$u$ and $v$ are constant, with $u\leq v$}
\ee
on $\R^n\times[0, \infty)\cap\{|x|+t\geq R\},$
for some constant $R\geq 0.$ Then $u\leq v$ on $\R^n\times[0, \infty).$
\end{theorem}

\begin{theorem}\label{thws4}
Assume that $u$ and $v$ are weak solution of \eqref{ws1} such that
\be\label{ws28}
\mbox{$u$ and $v$ are constant on $\R^n\times[0, \infty)\cap\{|x|+t\geq R\}$}
\ee
for some constant $R>0.$ Then
\be\label{ws29}
\max\limits_{0\leq t<\infty}\|u(\cdot, t)-v(\cdot, t)\|_{L^\infty(\R^n)}
=\|u(\cdot, 0)-v(\cdot, 0)\|_{L^\infty(\R^n)}.
\ee
\end{theorem}

\bigskip
\section{Existence of weak solutions}
\label{secew}
\setcounter{equation}{0}

\subsection{Solution of the approximate equations.}\label{ews1}
In this subsection, we are going to study general curvature flow in weak sense. Moreover, we are not going to restrict ourselves in the admissible cone $K$.
Let's define
\be\label{ew1}
\hf^n(\tau)=\inf\limits_{\lambda}\{f(\lambda)+Df(\lambda)\cdot(\tau-\lambda): \lambda\in K_n\,\, \mbox{and} \,\,\lambda_{\max}<n\},
\ee
where $K_n\subset K$ and $f\mid_{\partial K_n}=1/n.$ Then we have $\hf^n(\tau)=f(\tau)$ when $\tau\in K_n$ and $\tau_{\max}<n.$ Furthermore, $\hf^n$ is concave, Lipschitz and satisfying
$\hf^n_i>0$ in $\R^n.$

Now, let's consider
\be\label{ew2}
\begin{aligned}
&u^{\e, n}_t=\hF^n(\gamma^{\e ik}u^{\e, n}_{kl}\gamma^{\e lj} ),\,\,\mbox{in $\R^n\times[0, \infty)$}\\
&u^{\e, n}=g \,\,\mbox{on $\R^n\times\{t=0\}$},\\
\end{aligned}
\ee
where $\gamma^{\e ik}=\delta_{ik}-\frac{u_iu_k}{\e \sqrt{\e^2+|Du|^2}+\e^2+|Du|^2},$ and $\gamma^{\e ik}\gamma^{\e kj}\mid_{u}=\lt(\delta_{ij}-\frac{u_iu_j}{\e^2+|Du|^2}\rt).$ We also assume $\|g\|_{C^2(\R^n)}$ is bounded.
\begin{theorem}\label{thew1}
For each $0<\e<1$ there exists a unique bounded solution $\uen\in C^{2, 1}(\R^n\times[0, \infty))$ of \eqref{ew2}. In addition,
\be\label{ew3}
\sup\limits_{0<\e<1}\|\uen, D\uen, \uen_t\|_{L^{\infty}(\R^n\times[0, \infty))}\leq C\|g\|_{C^2(\R^n)},
\ee
and
\be\label{ew3'}
\|D^2\uen\|_{L^\infty(\R^n\times [0, \infty))}\leq C(\|g\|_{C^2(\R^n)}, \e, n)
\ee
\end{theorem}
\begin{proof}
1. For each $0<\sigma<1,$ consider the PDE
\be\label{ew4}
\begin{aligned}
&\us_t=\hF^n(\gamma^{\e ik}\us_{kl}\gamma^{\e lj})+\sigma \us_{ii},\,\mbox{in $\R^n\times [0, \infty)$}\\
&\us=g\,\,\mbox{on $\R^n\times\{t=0\},$}\\
\end{aligned}
\ee
Now let's denote
\[\hF^n(\gamma^{\e ik}\us_{kl}\gamma^{\e lj})=G(D^2\us, D\us),\]
then the linearized operator is
\be\label{ew5}
\mathfrak{L}=\frac{\partial}{\partial t}-\Gs^{ij}\partial_{ij}-\Gs^i\partial_i-\sigma\delta_{ij}\partial_{ij}.
\ee
It is easy to see that $\mathfrak{L}$ is uniformly parabolic. Thus, we obtain that
there exists a unique solution $\us$ satisfies
\be\label{ew6}
\|\us\|_{L^{\infty}(\R^n\times[0, \infty)}=\|g\|_{L^\infty(\R^n)}.
\ee

2. Now differentiating equation \eqref{ew4} with respect to $x_l$ we get
\be\label{ew7}
\mathfrak{L}\us_l=0,
\ee
which implies
\be\label{ew8}
\|D\us\|_{L^{\infty}(\R^n\times[0, \infty))}=\|Dg\|_{L^{\infty}(\R^n)}.
\ee
Similarly, we  have
\be\label{ew9}
\|\us_t\|_{L^{\infty}(\R^n\times[0, \infty))}\leq C\|D^2 g\|_{L^{\infty}(\R^n)}.
\ee

3. Since
\be\label{ew 10}
G^{ij}_{\e, n, \sigma}\xi_i\xi_j\geq C(n)\lt(1-\frac{L^2}{L^2+\e^2}\rt)|\xi|^2,
\ee
provided $\|Dg\|_{L^{\infty}(\R^{n})}\leq L.$
We deduce from \eqref{ew9} that we have bounds on the second derivatives of $\{\us\}$ which are uniform in $\sigma.$
In particular,
\be\label{ew11}
\|D^2\us\|_{L^{\infty}(\R^n\times[0, \infty))}\leq C(\e, n, \|g\|_{C^2(\R^n)}).
\ee

4. By Schauder estimates we conclude $\|\{\us\}_{0<\sigma<1}\|_{C^{2+1,1+1/2}(\R^n\times[0, \infty))}\leq C,$
where $C$ is independent of $\sigma.$ Thus, we have for each multi-index $\alpha,$ $|\alpha|\leq 2,$
\[D^\alpha \us\goto D^\alpha\uen\,\,\mbox{locally uniformly as $\sigma\goto 0,$}\]
for a $C^{2,1}$ function $\uen$ solving \eqref{ew2}.
\end{proof}

\subsection{Passing to the limit}\label{ews2}
\begin{theorem}\label{thew2}
Assume $g:\R^n\goto\R$ is continuous and satisfying that
\be\label{ew41'}
\mbox{$g$ is constant on $\R^n\cap\{|x|\geq S\}.$}
\ee
Then, there exists a weak solution $u$ of \eqref{ws1}, such that
\be\label{ew13'}
\mbox{$u$ is constant on $\R^n\times[0, \infty)\cap\{|x|+t\geq R\}.$}
\ee
\end{theorem}
\begin{proof}
1. Suppose temporarily $g$ is smooth. We can extract a subsequence
$\{u^{\e_k, n_k}\}_{k=1}^{\infty}\subset\{\uen\}_{0<\e\leq 1}$ so that,
$u^{\e_k, n_k}\goto u$ locally uniformly in $\R^n\times[0, \infty)$
for some bounded, Lipschitz function $u$.

2. We assert now that $u$ is a weak solution of
\be\label{ew14}
\begin{aligned}
&u_t=F(\gamma^{ik}u_{kl}\gamma^{lj} ),\,\,\mbox{in $\R^n\times[0, \infty)$}\\
&u=g \,\,\mbox{on $\R^n\times\{t=0\}$}.\\
\end{aligned}
\ee
For this, let $\phi\in C^{\infty}(\R^{n+1})\cap \Lambda(D)$ and suppose $u-\phi$ has a strictly local maximum
at a point $(x_0, t_0)\in\R^n\times(0, \infty)\cap D.$
As $\uk\goto u$ uniformly near $(x_0, t_0),$ $\uk-\phi$ has a local maximum at a point $(x_k, t_k)$
with $(x_k, t_k)\goto(x_0, t_0)\in D,$ as $k\goto\infty.$

Since $\uk$ are $C^{2, 1}$, we have
\[D\uk=D\phi,\,\,\uk_t=\phi_t,\,\, D^2\uk\leq D^2\phi\,\,\mbox{at $(x_k, t_k)$.}\]
Thus \eqref{ew2} implies that
\be\label{ew15}
\phi_t-\hF^{n_k}(\gamma^{\e_k ik}\phi_{kl}\gamma^{\e_k lj} )\leq 0,\,\,\mbox{at $(x_k, t_k).$}
\ee

Now suppose first that $|D\phi|\neq 0$ at $(x_0, t_0).$
Since $f(\kappa[\gamma^{ik}\phi_{kl}\gamma^{lj}])>0$ at $(x_0, t_0),$
we have for $k$ very large
\be\label{ew16}
\hF^{n_k}(\gamma^{\e_k ik}\phi_{kl}\gamma^{\e_k lj})=F(\gamma^{\e_k ik}\phi_{kl}\gamma^{\e_k lj})
\,\,\mbox{at $(x_k, t_k).$}
\ee
Let $k\goto \infty$ we get
\be\label{ew17}
\phi_t-F(\gamma^{ik}\phi_{kl}\gamma^{lj})\leq 0\,\,\mbox{at $(x_0, t_0)$.}
\ee

Next, assume instead $D\phi(x_0, t_0)=0.$ Set
\be\label{ew20}
\eta^k=\frac{D\phi}{(w^{\e_k}_\phi(\e_k+w^{\e_k}_\phi))^{1/2}},
\ee
where $w^{\e_k}_\phi=\sqrt{\e_k^2+|D\phi|^2}.$
Then \eqref{ew15} becomes
\be\label{ew21}
\phi_t\leq \hF^{n_k}\lt(\gamma^{ik}_{\eta^k}\phi_{kl}\gamma^{lj}_{\eta^k}\rt).
\ee
We may assume, upon passing to a subsequence and reindexing if necessary, that $\eta^k\goto\eta$
in $\R^n$ for some $|\eta|\leq 1.$
Following the above argument we get
\be\label{ew22}
\phi_t\leq F\lt(\gamma^{ik}_\eta\phi_{kl}\gamma^{lj}_\eta\rt)\,\,\mbox{at $(x_0, t_0).$}
\ee

If $u-\phi$ has a local maximum, but not necessarily a strict local maximum at $(x_0, t_0),$
we repeat the argument above by replacing $\phi(x, t)$ with
\be\label{ew23}
\tilde{\phi}(x, t)=\phi(x, t)+|x-x_0|^4+(t-t_0)^4.
\ee
Consequently, u is a weak subsolution. u is a weak supersolution follows analogously.

3. We want to verify that there exists a weak solution $u$ such that \eqref{ew13'} holds.
Upon rescaling as necessary, we may assume
\be\label{ew24}
|g|\leq 1\,\, \mbox{on $\R^n$,} g=0 \,\,\mbox{on $\R^n\cap\{|x|\geq1\}$.}
\ee
Consider now the auxiliary function
\be\label{ew25}
v(x, t)\equiv \varphi(|x|^2/2+c_0t),
\ee
where $c_0>0$ and
\be\label{ew26}
\varphi(s)\equiv\left\{
\begin{aligned}
&0\,\,(s\geq 2)\\
&(s-2)^3\,\,(0\leq s\leq 2).\\
\end{aligned}
\right.
\ee
Then
\be\label{ew27}
\varphi'(s)\equiv\left\{
\begin{aligned}
&0\,\,(s\geq 2)\\
&3(s-2)^2\,\,(0\leq s\leq 2),\\
\end{aligned}
\right.
\ee
and
\be\label{ew28}
\varphi''(s)\equiv\left\{
\begin{aligned}
&0\,\,(s\geq 2)\\
&6(s-2)\,\,(0\leq s\leq 2).\\
\end{aligned}
\right.
\ee
In particular, $|\varphi''(x)|\leq C(\varphi'(s))^{1/2} \,\,(s\geq 0).$
We have
\be\label{ew29}
\begin{aligned}
&v_t-\hF^n\lt(\gamma_v^{\e ik}v_{kl}\gamma_v^{\e lj}\rt)\\
&=c_0\varphi'-\hF^{n, ij}_v\lt(\delta_{ik}-\frac{v_iv_k}{w^\e_v(\e+w^\e_v)}\rt)v_{kl}\lt(\delta_{lj}-\frac{v_lv_j}{w^\e_v(\e+w^\e_v)}\rt)\\
&=c_0\varphi'-\hF^{n, ij}_v\lt(v_{il}-\frac{v_iv_kv_{kl}}{w^\e_v(\e+w^\e_v)}\rt)\lt(\delta_{lj}-\frac{v_lv_j}{w^\e_v(\e+w^\e_v)}\rt)\\
&=c_0\vp'-\hF^{n, ij}_v\lt(v_{ij}-\frac{v_iv_kv_{kj}}{w^\e_v(\e+w^\e_v)}-\frac{v_lv_jv_{il}}{w^\e_v(\e+w^\e_v)}
+\frac{v_iv_kv_lv_jv_{kl}}{(w^\e_v)^2(\e+w^\e_v)^2}\rt)\\
&=c_0\vp'-\hF^{n, ij}_v\lt(\vp'\delta_{ij}+\vp''x_ix_j\rt)+2\hF^{n, ij}_v\frac{v_iv_kv_{kj}}{w^\e_v(\e+w^\e_v)}
-\hF^{n, ij}_v\frac{v_iv_jv_kv_lv_{kl}}{(w^\e_v)^2(\e+w^\e_v)^2}\\
&=c_0\vp'-\vp'\sum\hF^{n, ii}_v-\vp''\hF^{n, ij}_vx_ix_j+2\hF^{n, ij}_v\frac{(\vp')^2x_ix_k(\vp'\delta_{kj}+\vp''x_kx_j)}{w^\e_v(\e+w^\e_v)}\\
&-\hF^{n, ij}_v\frac{(\vp')^4x_ix_jx_kx_l(\vp'\delta_{kl}+\vp''x_kx_l)}{(w^\e_v)^2(\e+w^\e_v)^2}\\
&=\vp'\lt[c_0-\hF^{n, ij}_v\lt(\delta_{ij}-2\frac{(\vp')^2x_ix_j}{w^\e_v(\e+w^\e_v)}
+\frac{(\vp')^4x_ix_j|x|^2}{(w^\e_v)^2(\e+w^\e_v)^2}\rt)\rt]\\
&-\vp''\hF^{n, ij}_v\lt[x_ix_j-2\frac{(\vp')^2x_ix_j|x|^2}{w^\e_v(\e+w^\e_v)}
+\frac{(\vp')^4x_ix_j|x|^4}{(w^\e_v)^2(\e+w^\e_v)^2}\rt]\equiv A+B,\\
\end{aligned}
\ee
where $w^\e_v=\sqrt{\e^2+|Dv|^2}=\sqrt{\e^2+|\vp'|^2|x|^2}.$
\be\label{ew30}
\begin{aligned}
A&=\vp'\lt[c_0-\hF^{n, ij}_v\lt(\delta_{ij}-\frac{x_ix_j}{|x|^2}\rt)
-\hF^{n, ij}_vx_ix_j\lt(\frac{1}{|x|^2}-2\frac{(\vp')^2}{w^\e_v(\e+w^\e_v)}
+\frac{(\vp')^4|x|^2}{(w^\e_v)^2(\e+w^\e_v)^2}\rt)\rt]\\
&\leq\vp'\lt[c_0-\hF^{n, ij}_v\lt(\delta_{ij}-\frac{x_ix_j}{|x|^2}\rt)\rt]\\
&\leq\vp'\lt[c_0-f\lt(\kappa[\delta_{ij}-\frac{x_ix_j}{|x|^2}]\rt)\rt],\\
\end{aligned}
\ee
the last inequality comes from the concavity assumption on $f$. We can see that when $c_0\leq f(0, 1, \cdots, 1),$ we get $A\leq 0.$
\be\label{ew31}
\begin{aligned}
|B|&=|\vp''|\hF^{n, ij}_vx_ix_j\lt[1-\frac{(\vp')^2|x|^2}{w^\e_v(\e+w^\e_v)}\rt]^2\\
&\leq C(n)|\vp''|\frac{\e^2}{(w^\e_v)^2}|x|^2.
\end{aligned}
\ee
Now if $|\vp'|\leq\e$ then $|B|\leq C(n)|\vp''|\leq C(n)\e^{1/2};$
if $|\vp'|\geq\e$ then $|B|\leq C(n)\frac{|\vp''|}{|\vp'|^2}\e^2\leq C(n)\frac{\e^2}{|\vp'|^{3/2}}\leq C(n)\e^{1/2}.$
Therefore, we get
\[|B|\leq C(n)\e^{1/2}.\]
Combining \eqref{ew29}-\eqref{ew31} yields
\be\label{ew32}
v_t-\hF^n\lt(\gamma^{\e ik}v_{kl}\gamma^{\e lj}\rt)\leq C(n)\e^{1/2}.
\ee
Therefore, let $\omega^{\e,n}\equiv v(x, t)-C(n)t\e^{1/2},$ then
$\omega^{\e, n}$ satisfies
\be\label{ew33}
\omega_t^{\e, n}-\hF^n\lt(\gamma^{\e ik}\omega^{\e, n}_{kl}\gamma^{\e lj}\rt)\leq 0.
\ee
Moreover,
\be\label{ew34}
\omega^{\e, n}(x, 0)=\vp(|x|^2/2)=0,\,\,\mbox{if $|x|\geq2$}
\ee
and
\be\label{ew35}
\omega^{\e, n}(x, 0)=\vp(|x|^2/2)\leq -1,\,\,\mbox{if $|x|\leq 1.$}
\ee

We see $\omega^{\e, n}(x, 0)\leq g$ on $\R^n\times\{t=0\},$ applying maximum principle we
deduce $\omega^{\e, n}\leq \uen$ in $\R^n\times[0, \infty)$ for any $0<\e<1.$ Let $(\e, n)=(\e_k, n_k),$
where $(\e_k, n_k)$ is a subsequence such that $\lim\limits_{k\goto\infty}C(n_k)\e_k^{1/2}=0$ and $u^{\e_k, n_k}\goto u.$
Then, sending $k\goto\infty$ we get
\be\label{ew36}
v(x, t)\leq u(x, t)
\ee
for all $x\in\R^n,\,\,t\geq 0.$ Thus, $u\geq 0$ if $|x|^2/2+c_0t\geq 2.$
Similarly, we let $\tilde{\omega}^{\e, n}=-\omega^{\e, n},$ and deduce
 $u\leq 0$ if $|x|^2/2+c_0t\geq 2.$ Assertions \eqref{ew13'} is proved.

4. Suppose $g$ satisfies \eqref{ew41'} but is only continuous. We select smooth
$\{g^k\}_{k=1}^\infty,$ satisfying \eqref{ew41'} for the same $S$, so that $g^k\goto g$
uniformly on $\R^n.$ Denote by $u^k$ the solution of $\eqref{ws1}$ constructed above with initial function $g^k.$
By Theorem \ref{thws4} we see $\lim\limits_{k\goto\infty}u^k=u$ exists uniformly on $\R^n\times[0, \infty).$
According to Theorem \ref{thws1}, u is a weak solution of \eqref{ws1}.
\end{proof}

\bigskip
\section{Consistency with classical motion by general curvature}
\label{seccw}
\setcounter{equation}{0}

In this section, we will check that our generalized evolution by general curvature agrees with the classical motion when the initial function $g(x)$ is properly chosen. Let us suppose for this section that $\Gamma_0$ is a smooth hypersurface, the connected boundary of a bounded open set $U\subset\R^n,$ and
$\kappa(A(\Gamma_0))\subset K.$ By standard short time existence theorem, we know that there exists a time $t^*>0$ and a family $\{\Sigma_t\}_{0\leq t<t^*}$ of smooth hypersurfaces evolving from $\Sigma_0=\Gamma_0$ according to classical motion by general curvature. In particular for each
$0\leq t<t^*,$ $\Sigma_t$ is diffeomorphic to $\Gamma_0,$ and it's a boundary of an open set $U_t$ diffeomorphic to $U_0\equiv U.$

\begin{remark}
\label{cwrmk1}
In \cite{ES1}, Evans and Spruck showed that the level set mean curvature flow does not depend upon the particular choice of initial function $g(x).$
However, this is not true in our case, since our approximation $\hat{f}^n(\kappa)$ is a "good approximation" to $f(\kappa)$ only when $\kappa\in K.$
\end{remark}

\begin{theorem}
\label{thcw1}
Let $\{\Sigma_t\}_{0\leq t<t^*}$ be a family of smooth hypersurfaces evolving from $\Sigma_0=\Gamma_0$ according to classical motion by general
curvature. Moreover, $\kappa[\Sigma_t]\in\ K,$ for all $0\leq t<t^*.$ Then there exists $\{\Gamma_t\}_{t\geq 0}$ such that $\Sigma_t=\Gamma_t\,\,(0\leq t<t^*),$ where $\{\Gamma_t\}_{t\geq 0}$ is the generalized evolution by general curvature.
\end{theorem}
\begin{proof}
1. Fix $0<t_0<t^*,$ and define then for $0\leq t\leq t_0$ the signed distance function
\be\label{cw1}
d(x, t)\equiv\left\{
\begin{aligned}
&-\text{dist}(x, \Sigma_t) \,\,\mbox{if $x\in U_t$,}\\
&\text{dist}(x, \Sigma_t)\,\,\mbox{if $x\in\R^n\setminus\bar{U}_t.$}
\end{aligned}
\right.
\ee
As $\Sigma\equiv \bigcup\limits_{0\leq t\leq t_0}\Sigma_t\times\{t\}$
is smooth, $d$ is smooth in the regions
\[Q^+\equiv\{(x, t)|0\leq d(x, t)\leq \delta_0,\,\,0\leq t\leq t_0\},\]
\[Q^-\equiv\{(x, t)|-\delta_0\leq d(x, t)\leq 0,\,\,0\leq t\leq t_0\}\]
for $\delta_0>0$ sufficiently small.

2. Now if $\delta_0>0$ is small enough, for each point $(x, t)\in Q^+$
there exists a unique point $y\in \Sigma_t$ such that $d(x, t)=|x-y|.$
Consider now near $(y, t)$ the smooth unit vector field $\nu\equiv Dd$
pointing from $\Sigma$ into $Q^+.$ Then
\be\label{cw2}
d_t(x, t)=f(\kappa_1, \cdots, \kappa_{n-1}, 0)
\ee
where $\kappa_1, \cdots, \kappa_{n-1}$ denote the principal curvature of $\Sigma_t$ at the point $y,$
calculated with respect to $-\nu.$
Moreover the eigenvalues of $D^2d(x, t)$ are
\[\lt\{\frac{\kappa_1}{1+\kappa_1d}, \cdots, \frac{\kappa_{n-1}}{1+\kappa_{n-1}d}, 0\rt\},\]
one can see that when $\delta_0$ is sufficiently small
\[\lt(\frac{\kappa_1}{1+\kappa_1d}, \cdots, \frac{\kappa_{n-1}}{1+\kappa_{n-1}d}, 0\rt)\subset K.\]
Therefore we have,
\be\label{cw3}
F\lt(\gamma^{ik}_dd_{kl}\gamma^{lj}_d\rt)=f\lt(\frac{\kappa_1}{1+\kappa_1d}, \cdots, \frac{\kappa_{n-1}}{1+\kappa_{n-1}d}, 0\rt),
\ee
and
\be\label{cw4}
d_t-F\lt(\gamma^{ik}_dd_{kl}\gamma^{lj}_d\rt)=\sum\limits_{i=1}^{n-1}f_i(\bar{\kappa}_i)\frac{\kappa_i^2}{1+\kappa_id}\cdot d.
\ee
Since $\kappa_i$ is uniformly bounded and $d\geq 0$ in $Q^+,$ let
\be\label{cw5}
\underline{d}\equiv\alpha e^{-\lambda t}d,
\ee
then $\ud$ satisfies
\be\label{cw6}
\ud_t-F\lt(\gamma^{ik}_{\ud}\ud_{kl}\gamma^{lj}_{\ud}\rt)\leq 0\,\,\mbox{ in $Q^+$,}
\ee
if $\lambda>0$ is chosen to be large enough and $\alpha>0$ will be determined later.
Furthermore, we have $|Dd|^2=|\nu|^2=1,$ $d_id_{ij}=0,$ and
$\lt(\frac{\kappa_1}{1+\kappa_1d}, \cdots, \frac{\kappa_{n-1}}{1+\kappa_{n-1}d}, 0\rt)\subset K $ in $Q^+.$
By \eqref{cw6} we get when $n$ large and $\e>0$ small,
\be\label{cw7}
\ud_t-\hF^n\lt(\gamma^{\e ik}_{\ud}{\ud}_{kl}\gamma^{\e lj}_{\ud}\rt)\leq 0\,\,\mbox{in $Q^+$.}
\ee
Therefore, we see that $\ud$ is a smooth subsolution of the approximate general curvature evolution equation in $Q^+.$

3. Choose any Lipschitz function $g:\R^n\goto \R$ so that $g(x)=d(x, 0)$
near $\Sigma_0,$ $\{x\in\R^n| g(x)=0\}=\Sigma_0,$ and $g(x)$ is a positive constant for large $|x|.$
For $0<\e<1$ and $n$ large, the approximating PDE has a continuous solution $u^{\e, n},$
which is $C^{2, 1}$ in $\R^n\times (0, \infty).$ Additionally, passing to a subsequence if necessary, we have
$u^{\e, n}\goto u$ locally uniformly. In the following we denote
\be\label{cw8}
\Gamma_t=\{x\in\R^n|u(x, t)=0\}, \,\,t\geq 0.
\ee
Now $u=g=\delta_0>0$ on $\{(x, 0)|\text{dist}(x, \Sigma_0)=\text{dist}(x, \Gamma_0)=\delta_0\};$
and as $u$ is continuous, we have
\be\label{cw9}
u\geq\delta_0/2>0\,\,\mbox{on $\{(x, t)| d(x, t)=\delta_0\}$}
\ee
for $0\leq t\leq t_0,$ provided $t_0>0$ is small enough. Hence we have
\be\label{cw10}
\uen\geq\delta_0/4\,\,\mbox{on $\{(x, t)|d(x, t)=\delta_0\}$}
\ee
for $0\leq t\leq t_0,$ $0<\e\leq\e_0,$ and $n> N,$ where $\e_0>0$ is sufficiently small and $N$ is sufficiently large.
Consequently, there exists $0<\alpha<1$ so that
\be\label{cw11}
\uen\geq\ud\,\,\mbox{on $\{(x, t)| d(x, t)=\delta_0\}$}
\ee
for $0\leq t\leq t_0,$ $0<\e\leq\e_0,$ and $n> N.$ Since $0<\alpha<1,$ we have
\be\label{cw12}
\uen\geq\ud,\,\,\mbox{on $\{(x, 0)|0\leq d(x, 0)\leq \delta_0\}$.}
\ee
Moreover, $g\geq 0$ implies $\uen\geq 0$ and so $\uen\geq\ud$ on $\{(x, t)|d(x, t)=0\}.$

4. Therefore we see that $\uen\geq\ud$ on the parabolic boundary of $Q^+.$
By the maximum principle we have $\uen\geq\ud$ in $Q^+.$ Let $(\e, n)\goto (0, \infty)$ we conclude
$u>0$ in the interior of $Q^+.$

Next, considering $\ud=\alpha e^{-\lambda t}d$ in the interior of $Q^-$ instead, by a similar argument we can show $u<0$ in $Q^-.$ Since $u>0$ in $Q^+$
and $u<0$ in $Q^-$ we have
\be\label{cw13}
\Gamma_t\subseteq\Sigma_t=\{x|d(x, t)=0\}\,\,(0\leq t\leq t_0).
\ee

5. Now, let $\Gamma_t=\{x\in\R^n|u(x, t)=0\}\,\,(t\geq 0).$
Since $g<0$ in $U_0$ we know by continuity that $u<0$ somewhere in $U_t,$
provided $0\leq t\leq t_0$ and $t_0$ is small. Similarly, $u>0$ somewhere in
$\R^n-\bar{U}_t$ for each $0\leq t\leq t_0.$ Fix any point and draw a smooth curve C in $\R^n,$
intersecting $\Sigma_t$ precisely at $x_0$ and connecting a point $x_1\in U_t,$ where $u(x_1, t)<0,$
to a pint $x_2\in \R^n-\bar{U}_t$ where $u(x_2, t)>0.$ As $u$ is continuous, we must have
$u(x, t)=0$ for some point $x$ on the curve C. However, from step 4 we know
\be\label{cw14}
\{x|u(x, t)=0\}\subseteq\Sigma_t.
\ee
Thus, $u(x_0, t)=0,$ which implies $\Gamma_t=\Sigma_t$ if $0\leq t\leq t_0.$

We have demonstrated that the classical motion $\{\Sigma_t\}_{0\leq t<t^*}$
and the generalized motion $\{\Gamma_t\}_{t\geq 0}$ agree at least on some short time interval $[0, t_0].$

6. Write
\be\label{cw15}
s\equiv\sup\limits_{0\leq t<t^*}\{t|\Gamma_{\tau}=\Sigma_{\tau}\,\, \mbox{for all $0\leq\tau\leq t$}\},
\ee
and suppose $s<t^*.$ Then $\Gamma_t=\Sigma_t$ for all $0\leq t<s,$ and so, applying the continuity of the solution $u$,
we have $\Sigma_s\subseteq\Gamma_s.$ On the other hand, if $x\in\R^n-\Sigma_s,$
then there exists $r>0$ such that $B(x, r)\subset\R^n-\Sigma_t$ for all $s-\e\leq t\leq s,$ $\e>0$
small enough. This implies that $x\not\in\Gamma_s.$ Therefore we have $\Gamma_s=\Sigma_s=\{x\in\R^n| u(x, s)=0\}.$
Let's denote $\hat{g}(x)=u(x, s),$ by continuity we know $\hat{g}<0$ in $U_s$ and $\hat{g}(x)>0$ in $\R^n\setminus\bar{U}_s.$
Moreover, by Theorem \ref {thew2} we know there exists a weak solution $\hat{u}$ of \eqref{ws1} such that
\[\hat{u}(x, t-s)=u(x, t)=\text{constant} \,\,\mbox{on $\R^n\times [0, \infty)\cap\{|x|+t\geq R\},$}\]
and
\[\hat{u}(x, 0)=u(x, s)=\hat{g}(x).\]
By the Comparison Theorem \ref {thws3} we also know that $\hat{u}(x, t-s)=u(x, t)$ for $t\geq s.$
On the other hand, we can apply steps 1-5 and deduce $\hat{\Gamma}_{t-s}=\Sigma_t$ for all
$s\leq t\leq s+s_0<t^*,$ if $s_0>0$ is small enough, which leads to a contradiction.
\end{proof}

From the above proof, we can see that the crutial point in our argument is to find an initial function $g(x)$ satisfies
$\kappa(A[\Gamma_t^g])\subset K,$ for $t\in[-\e_0, \e_0],$ where $\Gamma_t^g=\{x\in\R^n|g(x)=t\}.$
\bigskip

\section{Initial surface with positive general curvature}
\label{secci}
\setcounter{equation}{0}
In this section, we are going to assume that $\Gamma_0$ is a smooth connected hypersurface with $\kappa[A(\Gamma_0)]\subset K$, which is the boundary of
a bounded open set $U\subset\R^n.$ We will solve the general curvature equation \eqref{ew14} by separating variables.
Following the idea in \cite{ES1}, we will show that there exists a weak solution of \eqref{ew14} that can be represented as
\be\label{ci1}
u(x, t)\equiv v(x)+t,\,\,x\in U, t>0,
\ee
where $v$ is the unique weak solution of the stationary problem
\be\label{ci2}
F\lt(\gamma^{ik}v_{kl}\gamma^{lj}\rt)=1\,\,\mbox{in $U$,}
\ee
\be\label{ci3}
v=0\,\,\mbox{on $\partial U=\Gamma_0.$}
\ee
In this case, we have
\[\Gamma_t=\{x\in U| v(x)=-t\},\,\, t^*>t\geq 0,\]
and $\Gamma_t=\emptyset$ for $t>t^*\equiv\|v\|_{L^\infty(U)}.$
We informally interpret our PDE \eqref{ci2} as implying $\Gamma_t$ has positive general curvature for $0\leq t<t^*.$

Before carrying out the foregoing program rigorously, let's first give a definition of the weak solution to \eqref{ci2}.
\begin{defin}\label{cidf.1}
$v\in C(\bar{U})$ is a weak solution to \eqref{ci2} provided that for each $\phi\in \Lambda(D)\cap C^{\infty}(\R^n),$  if $v-\phi$ has a local maximum (minimum) at a point $x_0\in D\subset U,$ then
\be\label{ci.1'}
\lt\{\begin{aligned}
&F(\gamma^{ik}\phi_{kl}\gamma^{lj})\geq (\leq)1\\
&\mbox{at $x_0$ if $D\phi(x_0)\neq 0$}\\
\end{aligned}\rt.
\ee
and
\be\label{ci.2'}
\lt\{\begin{aligned}
&F(\gamma^{ik}_\eta\phi_{kl}\gamma^{lj}_\eta)\geq (\leq)1\,\,\mbox{at $x_0$ if $D\phi(x_0)= 0$}\\
&\mbox{for some $\eta\in\R$ with $|\eta|\leq 1.$}\\
\end{aligned}\rt.
\ee
\end{defin}

Equivalently we have
\begin{defin}\label{cidf.1'}
$v\in C(\bar{U})$ is a weak solution to \eqref{ci2} if whenever $x_0\in U$ and
\be\label{ci.3'}
v(x)\leq (\geq) v(x_0)+p\cdot(x-x_0)+\frac{1}{2}(x-x_0)^{\top}R(x-x_0)+o(|x-x_0|^2)\,\,
\mbox{as $x\goto x_0$}
\ee
for some $p\in\R^n$ and $R=\{r_{ij}\}\in S^{n\times n},$ then
\be\label{ci.4'}
F(\gamma^{ik}_pr_{kl}\gamma^{lj}_p)\geq (\leq) 1\,\,\mbox{at $x_0$ if $p\neq 0$ and $\kappa[\gamma^{ik}_pr_{kl}\gamma^{lj}_p]\in K$,}
\ee
and
\be\label{ci.5'}
F(\gamma^{ik}_\eta r_{kl}\gamma^{lj}_\eta)\geq (\leq) 1\,\,\mbox{at $x_0$ if $p=0$ for some $|\eta|\leq 1$ and $\kappa[\gamma^{ik}_pr_{kl}\gamma^{lj}_p]\in K$.}
\ee
\end{defin}

\begin{theorem}\label{thci0}There exists a unique weak solution $v$ of equations \eqref{ci2} and \eqref{ci3}. Furthermore,  there is constant $A>0$
so that
\be\label{ci6}
\begin{aligned}
-A\, \text{dist}(x, \Gamma_0)\leq v(x) &\leq0 \,\,\,(x\in\bar{U}),\\
|Dv(x)|&\leq A.
\end{aligned}
\ee
\end{theorem}
\begin{proof}
1. Similar to Section \ref{secew}, we will study the following approximate PDE instead:
\be\label{ci4}
\hF^n\lt(\gamma^{\e ik}v^\e_{kl}\gamma^{\e lj}\rt)=1\,\,\mbox{in $U$,}
\ee
\be\label{ci5}
v^{\e, n}=0\,\,\mbox{on $\partial U=\Gamma_0,$}
\ee
where $0<\e<1$ small and $n>N_0$ is a large integer. It's easy to see that
\be\label{ci14}
v^{\e, n}\leq 0,\,\,\mbox{in $U$.}
\ee
In the following, we will construct a lower barrier for
\eqref{ci4} of the form
\[\omega(x)=\lambda g(d(x))\,\,\,(\lambda\in\R, d(x)=\text{dist}(x, \Gamma_0), x\in U)\]
in a small neighborhood $V_{2\delta_0}\equiv\{x\in U| 0<d(x)<2\delta_0\}$ of $\Gamma_0.$ We choose $\delta_0>0$ small  such that in $V_{2\delta_0},$
$d(x)$ is smooth and
$\kappa[-D^2d]\subset K.$

By a straightforward calculation we get
\be\label{ci7}
\begin{aligned}
\hF^n[\gamma^{\e ik}_\omega\omega_{kl}\gamma^{\e lj}_\omega]&=\hF^{n, ij}\gamma^{\e ik}_\omega\omega_{kl}\gamma^{\e lj}_\omega\\
&=\hF^{n, ij}\gamma^{\e ik}_\omega(\lambda g'd_{kl}+\lambda g''d_kd_l)\gamma^{\e lj}_\omega\\
&=\lambda g'\hF^{n, ij}\gamma^{\e ik}_\omega d_{kl}\gamma^{\e lj}_\omega
+\lambda g''\hF^{n, ij}\gamma^{\e ik}_\omega d_kd_l\gamma^{\e lj}_\omega\\
&\geq \lambda \hf^n(\kappa[g'd_{kl}])+\frac{\e^2\lambda g''\hF^{n, ij}d_id_j}{\lambda^2g'^2+\e^2},\\
\end{aligned}
\ee
where in the last inequality we used the concavity of $\hF^n$ and $d_id_jd_{ij}\equiv 0.$

Now choose
\be\label{ci10}
g(t)=\log\lt(\frac{2\delta_0-t}{2\delta_0}\rt);
\ee
then $g(t)$ is convex on $[0, 2\delta_0)$ and satisfies
\be\label{ci11}
g(0)=0,\,\,g'\leq-1/(2\delta_0),\,\,g''=g'^2,\,\,g'(2\delta_0)=-\infty.
\ee
Therefore, by \eqref{ci7} we have,
\be\label{ci12}
\hF^n[\gamma^{\e ik}_\omega\omega_{kl}\gamma^{\e lj}_\omega]
\geq \lambda c_1>1
\ee
when $\lambda$ large.
Since $\frac{\partial\omega}{\partial\tilde{\nu}}=-\infty$ on $\{d=2\delta_0\},$
where $\tilde{\nu}$ denotes the exterior normal to $V_{2\delta_0},$ we can apply
the maximum principle to conclude that
\[v^{\e, n}>\omega\,\,\mbox{in $V_{2\delta_0}.$}\]
Thus
\be\label{ci13}
v^{\e, n}(x)\geq -Ad(x),\,\,x\in V_{\delta_0}\subset V_{2\delta_0}.
\ee
Equation \eqref{ci13} yields $|Dv^{\e, n}|\leq A$ on $\Gamma_0.$ By differentiating \eqref{ci4}
with respect to $x_l,$ we see that any derivative $v^{\e, n}_l$ achieves its maximum and minimum on
$\Gamma_0.$ Therefore, $|Dv^{\e, n}|\leq A$ in $U$ and in particular we have $v^{\e, n}\geq -Ad$ in $U.$

2. By step 1 we have
\be\label{ci15}
\sup\limits_{0<\e<1, n>N_0}\|v^{\e, n}\|_{C^{0, 1}}<\infty.
\ee
Hence we may extract a subsequence $\{v^{\e_k, n_k}\}_{k=1}^\infty$
so that when $k\goto\infty,$ $\e_k\goto 0,$ $n_k\goto\infty,$ and $v^{\e_k, n_k}\goto v$ uniformly in $\bar{U}.$
As in the proof of Theorem \ref{thew2}, we verify that $v$ is a weak solution of \eqref{ci2}.

3. The uniqueness of this weak solution $v$ will follow from the following comparison theorem.
\end{proof}

\begin{theorem}
\label{thci1}
Assume $u$ is a weak subsolution and $v$ is a weak supersolution of \eqref{ci2} and \eqref{ci3}, then we have that
\[u\leq v\,\,\mbox{in $U$.}\]
\end{theorem}

Before proving this theorem, let's first state the following lemma which will be used later.
\begin{lemma}
\label{lmci.1}
Let
\be\label{ci16}
\begin{aligned}
\omega^\e(x)&=\sup\limits_{y\in\R^n}\{\omega(y)-\e^{-1}|x-y|^2\},\\
\omega_\e(x)&=\inf\limits_{y\in\R^n}\{\omega(y)+\e^{-1}|x-y|^2\}.\\
\end{aligned}
\ee
Then there exists constants $A, B, C,$ depending only on $\|\omega\|_{L^\infty},$ such that for $\e>0$ the following hold:\\
(i) $\omega_\e\leq\omega\leq \omega^\e$ on $\R^n.$\\
(ii)$\|\omega^\e, \omega_\e\|_{L^{\infty}}(\R^n)\leq A.$\\
(iii) If $y\in \R^n$ and $\omega^\e(x)=\omega(y)-\e^{-1}|x-y|^2$ then $|x-y|\leq C\e^{1/2}\equiv\sigma(\e).$\\
(iv) $\omega^\e, \omega_\e\goto\omega$ as $\e\goto 0^+,$ uniformly on compact subset of $\R^n.$\\
(v) $\text{Lip}(\omega^\e), \text{Lip}(\omega_\e)\leq B/\e.$\\
(vi) The mapping $x\mapsto\omega^\e+\e^{-1}|x|^2$ is convex and the mapping
$x\mapsto\omega_\e-\e^{-1}|x|^2$ is concave.\\
(vii) Assume $\omega$ is a weak subsolution of \eqref{ci2} in $U,$ then $\omega^\e$ is a weak subsolution in $U_{\sigma(\e)}\subset U,$
where $U_{\sigma(\e)}=U\setminus\{x\in U|\text{dist}(x, \partial U)>\sigma(\e)\},$
and $U_{\sigma(\e)}\goto U$ as $\e\goto 0^+.$ Similarly, if $\omega$ is a weak supersolution of \eqref{ci2}, then $\omega_\e$ is a weak supersolution in
$U_{\sigma(\e)}.$\\
(viii) $\omega^\e$ is twice differentiable a.e. and satisfies
\[1\leq F\lt(\gamma^{ik}\omega^\e_{kl}\gamma^{lj}\rt)\]
at each point of twice differentiability in $U_{\sigma(\e)}$ where $D\omega^\e\neq 0.$ Similarly, we have
\[1\geq F\lt(\gamma^{ik}\omega_{\e kl}\gamma^{lj}\rt)\]
at each point of twice differentiability in $U_{\sigma(\e)}$ where $D\omega_\e\neq 0.$
\end{lemma}
The proof of this Lemma is similar to Lemma 3.1 of \cite{ES1}; we leave it to the reader.

\begin{proof} ( proof of theorem \ref{thci1})
In this proof, we extend $u(x)$ and $v(x)$ to $\R^n$ by letting $u(x)=v(x)=0$ on $\R^n\setminus U.$
We are going to prove this theorem using proof by contradiction.

1.  If $u\not\leq v,$ then we would have
\be\label{ci17}
\max\limits_{x\in \R^n}(u-v)\equiv u(x_0)-v(x_0)\equiv a>0,
\ee
for some $x_0\in U.$
Fix $\e>0$ small; we get that
\be\label{ci18}
\max\limits_{x\in \R^n}(u^\e-v_\e)\equiv u^\e(x_1)-v_\e(x_1)\geq a>0,
\ee
for some $x_1\in U.$

2. Given $\delta>0$ small, define
\be\label{ci19}
\max\limits_{x\in \bar{U}, y\in\bar{U}-x}\Phi(x, y)\equiv u^\e(x+y)-v_\e(x)-\delta^{-1}|y|^4+\e|x+y-x_1|^2.
\ee
Then it's easy to see that
\be\label{ci20}
\max\limits_{x\in \bar{U}, y\in\bar{U}-x}\Phi(x, y)\geq a>0.
\ee
Suppose $\Phi(x_2, y_2)=\max\limits_{x\in \bar{U}, y\in\bar{U}-x}\Phi(x, y).$ Then we have that
$|y_2|\leq C\delta^{1/4}.$
Moreover, it's easy to see that $x_2\in U_{\sigma(\e)}.$
According to Lemma \ref{lmci.1} (vii), when $\delta>0$ is small enough,  we have $u^\e$ is a weak subsolution of \eqref{ci2} near $x_2+y_2$ and
$v_\e$ is a weak supersolution of \eqref{ci2} near $x_2.$

3. We now demonstrate that $y_2\neq 0.$ If $y_2=0,$ then we would have
\[u^\e(x_2)-v_\e(x_2)+\e|x_2-x_1|^2\geq u^\e(x_2+y)-v_\e(x_2)-\delta^{-1}|y|^4+\e|x_2+y-x_1|^2,\]
which yields,
\be\label{ci21}
u^\e(x_2+y)\leq u^\e(x_2)+\delta^{-1}|y^4|+2\e\lt<x_1-x_2, y\rt>.
\ee
Since $u^\e$ is a weak subsolution of \eqref{ci2} near $x_2,$ by Definition \ref{cidf.1'} we obtain $1\leq 0.$ This leads to a contradiction.

4. Since $\Phi$ attains its local maximum at $(x_2, y_2),$ there exists a sequence $(x^k, y^k)\goto (x_2, y_2)$
as $k\goto\infty.$ Moreover, $\Phi, u^\e$ and $v_\e$ are twice differentiable at $(x^k, y^k).$ Then we have that
\[D_{x, y}\Phi(x^k, y^k)\goto 0,\] and
\[D^2_{x, y}\Phi(x^k, y^k)\leq o(1)I_{2n}.\]
Furthermore,
\be\label{ci22}
D_x\Phi=Du^\e(x^k+y^k)-Dv_\e(x^k)+\e D|x^k+y^k-x_1|^2=p^k-\bar{p}^k,
\ee
and
\be\label{ci23}
D_y\Phi=p^k-\frac{4}{\delta}|y^k|^2y^k,
\ee
where $p^k=Du^\e(x^k+y^k)+\e D|x^k+y^k-x_1|^2,$ $\bar{p}^k=Dv_{\e}(x^k).$
Let $p\equiv \frac{4}{\delta}|y_2|^2y_2,$ then we have $p^k, \bar{p}^k\goto p\neq 0.$
Moreover, we have that
\[D^2_x\Phi\equiv R^k-\bar{R}^k+2\e I_n, R^k-\bar{R}^k\leq \e_kI_n-2\e I_n,\,\, \mbox{where $\e_k\goto 0$ as $k\goto\infty$}.\]
This together with Lemma \ref{lmci.1} (vi) yields
\[-C(\e)I_n\leq R^k\leq\bar{R}^k+\e_k I_n-2\e I_n\leq C(\e)I_n.\]
 Consequently we may suppose that
$R^k\goto R,$ $\bar{R}^k\goto \bar{R}.$
By our assumption we get
\[R-\bar{R}<-2\e I_n,\]
\[1\leq F\lt(\gamma^{is}_{p^k}r^k_{sl}\gamma^{lj}_{p^k}\rt),\]
and
\[1\geq F\lt(\gamma^{is}_{\bar{p}^k}\bar{r}^k_{sl}\gamma^{lj}_{\bar{p}^k}\rt),\]
for all $k.$ Let $k\goto \infty$ we obtain that
\[0\leq F^{ij}\gamma^{is}_p(r-\bar{r})_{sl}\gamma^{lj}_{p}<0,\]
leads to a contradiction.
\end{proof}

\bigskip

\section{Non-collapsing result for general curvature flow}
\label{secnc}
\setcounter{equation}{0}

\subsection{Non-collapsing result}
In this subsection, we will prove the Ben-Andrews' Non-collapsing result (see \cite{A, ALM}) under the elliptic setting. Our approach is inspired by
the perturbation idea in \cite{CS}, where it is used to prove convexity properties of solutions to certain types of elliptic equations.
We hope that this approach can be generalized to a wider class of elliptic equations.

Let's first recall our equation
\be\label{nc1}
\left\{\begin{aligned}
|\nabla u|F(A[\Gamma_t])&=1,\,\,\mbox{in $\Omega\subset\R^{n+1}$}\\
u&=0,\,\,\mbox{on $\partial\Omega$},\\
\end{aligned}\rt.
\ee
where $F(A[\Gamma_t])=f(\kappa[\Gamma_t]),$ $\Gamma_t=\{x\in\Omega\mid u(x)=-t\}.$

The main result of this section is the following:
\begin{theorem}\label{thnc1}
Let $\Omega$ be a compact manifold, and $\partial\Omega$ satisfies the $\delta$-Andrews' condition: at each point of $x\in\partial\Omega$
there exists a ball in $\Omega$ with radius $\delta/F_x,$ where $\delta>0$ is fixed, and $F_x$ is $F(A[\partial\Omega])$ at the point $x.$ Let $u$ be a solution of equation
\eqref{nc1}. Then each level set of $u$ satisfies the $\delta$-Andrews condition.
\end{theorem}
\begin{proof}
Now we are going to define an admissible set $\mathcal{A}$ as following:
\[\mathcal {A}=\{(x, y):y\in B_r(C_x), r=\delta/F_x, C_x=x+r\nu_x,\,\,\mbox{and $x\in\Omega$}\},\]
where $\nu_x=-\frac{\nabla u}{|\nabla u|}$ is the interior normal to the level set
$\Gamma_{u(x)}=\{y\in\Omega: u(y)=u(x)\}.$
To prove the level set $\Gamma_t$ of equation \eqref{nc1} is non-collapsing is equivalent to prove
\be\label{nc7}
\inf\limits_{(x, y)\in\mathcal{A}}U(x, y)=u(x)-u(y)\geq 0.
\ee
We will prove \eqref{nc7} by contradiction. In the following we assume there
exists $(x_0, y_0)\in\Omega,$ such that
\be\label{nc8}
\inf\limits_{(x, y)\in\mathcal{A}}=U(x_0, y_0)<0.
\ee
Let $x_s=x_0+sF_{x_s}\nu_{x_s}$ and $y_s=y_0+sF_{y_s}\nu_{y_s}.$ Then we have that
\be\label{nc9}
\begin{aligned}
u(x_0)&=u(x_s-sF_{x_s}\nu_{x_s})\\
&=u(x_s)-sF_{x_s}\nabla u(x_s)\cdot\nu_{x_s}+O(s^2)\\
&=u(x_s)+s+O(s^2).\\
\end{aligned}
\ee
Similarly, we get that
\be\label{nc9'}
u(y_0)=u(y_s-sF_{y_s}\nu_{y_s})=u(y_s)+s+O(s^2).
\ee
Since $\|y-C_x\|^2-\lt(\frac{\delta}{F_x}\rt)^2=\frac{2}{F_x}Z(x, y),$
where
\be\label{nc10}
Z(x, y)=\frac{F_x}{2}\lt<x-y, x-y\rt>+\delta\lt<x-y, \nu_x\rt>.
\ee
It's easy to see that $Z(x_s, y_s)\geq 0$. Otherwise, by \eqref{nc9} and \eqref{nc9'}, we can perturb $(x_0, y_0)$
to some $(x, y)\in\mathcal{A}$ such that $U(x, y)<U(x_0, y_0).$
Therefore,
\be\label{nc11}
\frac{\partial}{\partial s}Z(x_s, y_s)\mid_{s=0}\, =0.
\ee
Differentiating $u(x_s)=u(x_0)-s+O(s^2)$ with respect to $s$ we get
\be\label{nc12}
\nabla u\cdot\frac{dx_s}{ds}\mid_{s=0}\,=-1.
\ee
Thus,
\be\label{nc13}
\frac{d}{ds}x_s\mid_{s=0}\,=F_{x_0}\nu_{x_0},
\ee
\be\label{nc14}
\frac{d}{ds}\tau^i_{x_s}\mid_{s=0}\,=(\nabla_iF_{x_0})\nu_{x_0}-F_{x_0}h^l_i\tau^l_{x_0},
\ee
and
\be\label{nc15}
\frac{d}{ds}\nu_{x_s}\mid_{s=0}\,=-\nabla F_{x_0}=-\nabla_iF_{x_0}\tau^i_{x_0}.
\ee
Differentiating equation \eqref{nc15} we have that
\be\label{nc16}
\frac{d}{ds}\nu^i_{x_s}\mid_{s=0}\,=-\lt(\nabla_i\nabla_j F_{x_0}\rt)\tau^j_{x_0}
-\lt(\nabla_iF_{x_0}\rt)h_{ij}\nu_{x_0}.
\ee
On the other hand,
\be\label{nc17}
\nu^i_{x_s}\mid_{s=0}\,=-h^l_i\tau_l,
\ee
which implies that
\be\label{nc18}
\frac{d}{ds}\nu^i_{x_s}\mid_{s=0}\,=-\lt(\frac{d}{ds}h^l_i\mid_{s=0}\rt)\tau_l
-h^l_i\lt(\nabla_lF_{x_0}\rt)\nu_{x_0}+F_{x_0}h^l_ih^r_l\tau^r_{x_0}.
\ee
Hence,
\be\label{nc19}
\frac{d}{ds}h^j_i\mid_{s=0}=\nabla_i\nabla_jF_{x_0}+F_{x_0}h^l_ih_l^j
\ee
and
\be\label{nc20}
\frac{d}{ds}F_{x_s}\mid_{s=0}=F^{ij}_{x_0}\nabla_i\nabla_jF_{x_0}+F^{ij}_{x_0}h^l_ih_l^jF_{x_0}.
\ee

Now, let $d=\|x_0-y_0\|$ and $w=\frac{x_0-y_0}{d}.$ We have that
\be\label{nc21}
\begin{aligned}
\lt.\frac{dZ}{ds}\rt|_{s=0}&=\frac{d^2}{2}\frac{d}{ds}F_{x_s}+F_{x_0}\lt<\frac{dx_s}{ds}-\frac{dy_s}{ds}, dw\rt>\\
&+\delta\lt<\frac{dx_s}{ds}-\frac{dy_s}{ds}, \nu_{x_0}\rt>+\delta\lt<x_0-y_0, \frac{d\nu_{x_s}}{ds}\rt>\\
&=\frac{d^2}{2}[F_{x_0}^{ij}\nabla_i\nabla_jF_{x_0}+F^{ij}_{x_0}h^l_ih_l^jF_{x_0}]
+F_{x_0}\lt<F_{x_0}\nu_{x_0}-F_{y_0}\nu_{y_0}, dw\rt>\\
&+\delta\lt<F_{x_0}\nu_{x_0}-F_{y_0}\nu_{y_0}, \nu_{x_0}\rt>+\delta\lt<dw, -\nabla F_{x_0}\rt>=0.\\
\end{aligned}
\ee
Therefore,
\be\label{nc22}
\begin{aligned}
\delta d\lt<w, \nabla F_{x_0}\rt>&=\frac{d^2}{2}[F_{x_0}^{ij}\nabla_i\nabla_jF_{x_0}+F^{ij}_{x_0}h^l_ih_l^jF_{x_0}]\\
&+F_{x_0}\lt<F_{x_0}\nu_{x_0}-F_{y_0}\nu_{y_0}, dw\rt>+\delta F_{x_0}-\delta F_{y_0}\lt<\nu_{x_0}, \nu_{y_0}\rt>.\\
\end{aligned}
\ee
Next, we are going to consider the perturbation in the tangential direction.
Let $x^\e=x_0+\e\tau_{x_0}$ and $y^{\eta}=y_0+\eta\tau_{y_0}.$ Then we have that
\be\label{nc23}
u(y^\eta)-u(x^\e)=u(y_0)-u(x_0)+O(\e^2)+O(\eta^2).
\ee
Similar to previous arguments, we have that $Z(x^\e, y^\eta)\geq 0$ in a neighborhood of $(x_0, y_0)$
on $T_{x_0}\Gamma_{u(x_0)}\times T_{y_0}\Gamma_{u(y_0)}.$

Now choose local orthonormal coordinates on
$\Gamma_{u(x_0)}\times\Gamma_{u(y_0)}$ at $(x_0, y_0),$ such that
$\partial_{x_0}^i=\partial_{y_0}^i$ for $i=1, \cdots, n-1$ and $\partial_{x_0}^n,$ $\partial_{y_0}^n$ are coplanar with
$\nu_{x_0},$ $\nu_{y_0}.$ Also denote $\theta$ as the angle between $-w$ and $\nu_{x_0}.$

We are going to estimate
$Z(x^\e, y^\eta)-Z(x_0, y_0).$ For simplicity, we only compute the case when $x^\e=x_0+\e\txn$ and
$y^\eta=y_0+\eta\tyn;$ the general case can be computed in the same way.
First, let's assume that
\[F_{x^\e}=F_{x_0}+2f_1\e+2f_2\e^2+o(\e^2),\]
and
\[\nu_{x^\e}=(1-b_1\e^2)\nu_{x_0}+b_2\e\txn+P_{\text{span}\{\nu_{x_0}, \txn\}^\bot}(\nu_{x^\e}).\]
Then, we have that
\be\label{nc24}
\begin{aligned}
&Z(x^\e, y_0)-Z(x_0, y_0)\\
&=\lt(\frac{F_{x^\e}}{2}-\frac{F_{x_0}}{2}\rt)d^2+F_{x^\e}\e\lt<dw, \txn\rt>+\frac{F_{x^\e}}{2}\e^2\\
&+\delta\lt<dw, \nu_{x^\e}-\nu_{x_0}\rt>+\delta\e\lt<\txn, \nu_{x^\e}\rt>\\
&=(f_1\e+f_2\e^2)d^2-\lt(F_{x_0}\e+2f_1\e^2\rt)d\sin\theta\\
&+\frac{F_{x^\e}}{2}\e^2+\delta d\lt<w, -b_1\e^2\nu_{x_0}+b_2\e\txn\rt>+\delta\e^2\lt<\txn, b_2\txn\rt>+o(\e^2)\\
&=f_1\e d^2+f_2d^2\e^2-F_{x_0}d\sin\theta\e-2f_1d\sin\theta\e^2\\
&+\frac{F_{x^\e}}{2}\e^2+\delta db_1\cos\theta\e^2-\delta b_2d\sin\theta\e+\delta b_2\e^2+o(\e^2).\\
\end{aligned}
\ee
Since $Z(x^\e, y_0)-Z(x_0, y_0)\geq 0$ we get that
\be\label{nc1'}
f_1d^2-F_{x_0}d\sin\theta-\delta b_2d\sin\theta=0,
\ee
and
\be\label{nc2'}
f_2d^2-2f_1d\sin\theta+\frac{F_{x^\e}}{2}+\delta db_1\cos\theta+\delta b_2\geq 0.
\ee
Now consider
\be\label{nc3'}
\begin{aligned}
&Z(x^\e, y^\eta)-Z(x_0, y_0)\\
&=\lt(\frac{F_{x^\e}}{2}-\frac{F_{x_0}}{2}\rt)d^2+F_{x^\e}d\lt<w, \e\txn-\eta\tyn\rt>\\
&+\frac{F_{x^\e}}{2}\lt<\e\txn-\eta\tyn, \e\txn-\eta\tyn\rt>+\delta\lt<dw, \nue-\nu_{x_0}\rt>+\delta\lt<\e\txn-\eta\tyn, \nue\rt>\\
&=\lt(f_2d^2-2f_1d\sin\theta+\delta db_1\cos\theta+\delta b_2\rt)\e^2
+\lt(\frac{2f_1\delta}{F_{x_0}}\sin2\theta+\delta b_2\cos2\theta\rt)\e\eta\\
&+\frac{F_{x^\e}}{2}\lt<\e\txn-\eta\tyn, \e\txn-\eta\tyn\rt>+\text{higher order terms}\\
&\geq\lambda\lt<\alpha(x-x_0)-(y-y_0), \alpha(x-x_0)-(y-y_0)\rt>,\\
\end{aligned}
\ee
where the last inequality comes from $Z(x^\e, y^\eta)-Z(x_0, y_0)\geq 0.$ Moreover, we have that $\lambda>0$ depends on $x_0,$
and $\alpha$ depends on $\delta,$ $x_0,$ and $\theta.$

Therefore, in a small neighborhood of $(x_0, y_0)$
on $T_{x_0}\Gamma_{u(x_0)}\times T_{y_0}\Gamma_{u(y_0)},$ there is a small $\lambda>0$ such that
\be\label{nc25}
\hZ(x, y)=Z(x, y)-\lambda\lt<\alpha(x-x_0)-(y-y_0), \alpha(x-x_0)-(y-y_0)\rt>\geq 0,
\ee
and $\hZ(x_0, y_0)=0.$

In the following, all calculations are done at the point $(x_0, y_0).$ Also, these calculations are similar to \cite{ALM}. For reader's convenience we include it here.

Let's compute the first and second derivatives of $\hZ(x, y).$ First, differentiating equation \eqref{nc25} with respect to $y$ we get that
\be\label{nc26}
\begin{aligned}
\frac{\partial\hZ}{\partial y_i}&=\frac{\partial Z}{\partial y_i}-2\lambda\lt<-\partial^i_y, \alpha(x-x_0)-(y-y_0)\rt>\\
&=F_x\lt<-\partial^i_y, dw\rt>+\delta\lt<-\partial^i_y, \nu_x\rt>\\
&=\lt<-\partial^i_y, F_xdw+\delta\nu_x\rt>=0,
\end{aligned}
\ee
which implies $F_{x_0}dw+\delta\nu_{x_0}\parallelsum \nu_{y_0}.$

By a direct calculation we obtain that
\be\label{nc27}
\begin{aligned}
\|F_{x_0}dw+\delta\nu_{x_0}\|^2&=F^2_{x_0}d^2+2\delta F_{x_0}d\lt<w, \nu_{x_0}\rt>+\delta^2\\
&=F^2_{x_0}\cdot\frac{2\delta}{F_{x_0}}\cos\theta\cdot d-2\delta F_{x_0}d\cos\theta+\delta^2\\
&=\delta^2.
\end{aligned}
\ee

Combining \eqref{nc26} and \eqref{nc27} we get
\be\label{nc28}
F_{x_0}dw+\delta\nu_{x_0}=\delta\nu_{y_0}.
\ee

Next, we compute the first derivative of $\hZ(x, y)$ with respect to $x,$
\be\label{nc28}
\begin{aligned}
&\frac{\partial\hZ}{\partial x_i}=\frac{\partial Z}{\partial x_i}-2\lambda\alpha\lt<\partial_x^i, \alpha(x-x_0)-(y-y_0)\rt>\\
&=\frac{d^2}{2}\nabla_iF_{x_0}+F_{x_0}\lt<\partial^i_x, dw\rt>-\delta\lt<dw, h^x_{ip}\partial^p_x\rt>=0.\\
\end{aligned}
\ee

Finally, we compute the second derivative of $\hZ(x, y).$ Differentiating equation \eqref{nc25} with respect to $x$ twice we have that
\be\label{nc29}
\begin{aligned}
\frac{\partial^2\hZ}{\partial x_i\partial x_j}&=\frac{d^2}{2}\nabla_i\nabla_j F_{x_0}+\nabla_iF_{x_0}\lt<\partial_x^j, dw\rt>
+\nabla_jF_{x_0}\lt<\partial^i_x, dw\rt>\\
&+F_{x_0}\lt<h^x_{ij}\nu_x, dw\rt>+F_{x_0}\lt<\partial^i_x, \partial^j_x\rt>-\delta\lt<\partial^j_x, h^x_{ip}\partial_x^p\rt>\\
&-\delta\nabla_jh^x_{ip}\lt<dw, \partial_x^p\rt>-\delta\lt<dw, h^x_{ip}h^x_{pj}\nu_x\rt>-2\alpha^2\lambda\delta_{ij}\\
&=\frac{d^2}{2}\nabla_{ij}F_{x_0}+2d\nabla_iF_{x_0}\lt<\partial_x^j, w\rt>+F_{x_0}h^x_{ij}d\lt<w, \nu_x\rt>\\
&-\delta h^x_{ij}-\delta d\nabla_jh^x_{ip}\lt<w, \partial^p_x\rt>-\delta dh^x_{ip}h^x_{pj}\lt<w, \nu_x\rt>+F_{x_0}\delta_{ij}-2\alpha^2\lambda\delta_{ij}.\\
\end{aligned}
\ee
Differentiating equation \eqref{nc25} with respect to $x$ and $y$ we have that
\be\label{nc30}
\begin{aligned}
\frac{\partial^2\hZ}{\partial x_i\partial y_j}&=
\lt<dw, -\partial^j_y\rt>\nabla_i F_{x_0}-F_{x_0}\lt<\partial^i_x, \partial^j_y\rt>
-\delta\lt<-\partial^j_y, h^x_{ip}\partial^p_x\rt>+2\alpha\lambda\lt<\partial_{x_0}^i, \partial_{y_0}^j\rt>\\
&=-d\nabla_iF_{x_0}\lt<w, \partial^j_y\rt>-F_{x_0}\lt<\partial^i_x, \partial^j_y\rt>+\delta h^x_{ip}\lt<\partial^j_y, \partial^p_x\rt>+2\alpha\lambda\lt<\partial_{x_0}^i, \partial_{y_0}^j\rt>.\\
\end{aligned}
\ee
Differentiating equation \eqref{nc25} with respect to $y$ twice we have that
\be\label{nc31}
\begin{aligned}
\frac{\partial^2\hZ}{\partial y_i\partial y_j}&=\frac{\partial}{\partial y_i}
\lt\{F_x\lt<-\partial^j_y, dw\rt>+\delta\lt<-\partial^j_y, \nu_x\rt>\rt\}-2\lambda\delta_{ij}\\
&=F_x\lt<-h^y_{ij}\nu_y, dw\rt>+F_x\lt<-\partial^j_y, -\partial^i_y\rt>+\delta\lt<-h^y_{ij}\nu_y, \nu_x\rt>-2\lambda\delta_{ij}\\
&=F_{x_0}\delta_{ij}-F_{x_0}h^y_{ij}d\lt<w, \nu_y\rt>-\delta h^y_{ij}\lt<\nu_x, \nu_y\rt>-2\lambda\delta_{ij}.\\
\end{aligned}
\ee

Therefore,
\be\label{nc32}
\begin{aligned}
&F^{ij}_{x_0}\lt(\frac{\partial^2\hZ}{\partial x_i\partial x_j}+2\lt<\partial^i_x, \partial^j_y\rt>\frac{\partial^2\hZ}{\partial x_i\partial y_j}\rt.
\lt.+\frac{\partial^2\hZ}{\partial y_i\partial y_j}\rt)\\
&=\frac{d^2}{2}F^{ij}_{x_0}\nabla_{ij}F_{x_0}+2dF^{ij}_{x_0}\nabla_iF_{x_0}\lt<\partial_x^j, w\rt>+F^2_{x_0}d\lt<w, \nu_x\rt>-\delta F_{x_0}\\
&-\delta dF^{ij}_{x_0}\nabla_jh^x_{ip}\lt<w, \partial^p_x\rt>-\delta dF^{ij}_{x_0}h^x_{ip}h^x_{pj}\lt<w, \nu_x\rt>
+F_{x_0}\sum\limits_iF^{ii}_{x_0}\\
&+2\lt<\partial^i_x, \partial^j_y\rt>\lt(-dF^{ij}_{x_0}\nabla_iF_{x_0}\lt<w, \partial^j_y\rt>-F_{x_0}F^{ij}_{x_0}\lt<\partial^i_x, \partial^j_y\rt>\rt.\\
&\lt.+\delta F^{ij}_{x_0}h^x_{ip}\lt<\partial^j_y, \partial^p_x\rt>\rt)+F_{x_0}\sum\limits_iF^{ii}_{x_0}-F_{x_0}F^{ij}_{x_0}h^y_{ij}\lt<w, \nu_y\rt>d\\
&-\delta F^{ij}_{x_0}h^y_{ij}\lt<\nu_x, \nu_y\rt>-2\lambda\sum\limits_iF^{ii}_{x_0}
-2\lambda\alpha^2\sum\limits_iF^{ii}_{x_0}+4\alpha\lambda\sum\limits_{i=1}^{n-1}F^{ii}_{x_0}+4\alpha\lambda\lt<\txn, \tyn\rt>^2F^{nn}_{x_0}\\
&=\frac{d^2}{2}F^{ij}_{x_0}\nabla_{ij}F_{x_0}+2dF^{nn}_{x_0}\lt<\nabla_nF_{x_0}, w\rt>+F^2_{x_0}d\lt<w, \nu_x\rt>-\delta F_{x_0}\\
&-\lt(\frac{d^2}{2}F^{ij}_{x_0}\nabla_{ij}F_{x_0}+\frac{d^2}{2}F^{ij}_{x_0}h^l_ih_{lj}F_{x_0}+F_{x_0}^2d\lt<\nu_{x_0}, w\rt>\rt.\\
&-F_{x_0}F_{y_0}d\lt<\nu_{y_0}, w\rt>+\delta F_{x_0}-\delta F_{y_0}\lt<\nu_{x_0}, \nu_{y_0}\rt>\bigg)\\
&-\delta dF^{ij}_{x_0}h^x_{ip}h^x_{pj}\lt<w, \nu_x\rt>+2F_{x_0}\sum\limits_iF^{ii}_{x_0}\\
&-2d\lt<\partial^n_x, \partial^n_y\rt>F^{nn}_{x_0}\nabla_nF_{x_0}\lt<w, \partial^n_y\rt>-2F_{x_0}F^{ij}_{x_0}\lt<\partial^i_x, \partial^j_y\rt>^2\\
&+2\delta F^{ij}_{x_0}h^x_{ip}\lt<\partial^j_y, \partial^p_x\rt>\lt<\partial^i_x, \partial^j_y\rt>-F_{x_0}F^{ij}_{x_0}h^y_{ij}d\lt<w, \nu_y\rt>\\
&-\delta F^{ij}_{x_0}h^y_{ij}\lt<\nu_x, \nu_y\rt>
-2\lambda\sum\limits_iF^{ii}_{x_0}-2\lambda\alpha^2\sum\limits_iF^{ii}_{x_0}
+4\alpha\lambda\sum\limits_{i=1}^{n-1}F^{ii}_{x_0}+4\alpha\lambda\lt<\txn, \tyn\rt>^2F^{nn}_{x_0}.\\
\end{aligned}
\ee
Here we used equation \eqref{nc22}.
From equation \eqref{nc28} we have that

\be\label{nc33}
\begin{aligned}
\nabla_nF_{x_0}&=\frac{2\delta}{d^2}\lt<dw, h^x_{nn}\partial^n_x\rt>-\frac{2F_{x_0}}{d^2}\cdot d\lt<\partial^n_x, w\rt>\\
&=\frac{2\delta}{d}h^x_{nn}\lt<w, \partial^n_x\rt>-\frac{2F_{x_0}}{d}\lt<w, \partial^n_x\rt>,\\
\end{aligned}
\ee
 which yeilds,
\be\label{nc34}
2d\lt<\nabla_nF_{x_0}, w\rt>=4\delta h^x_{nn}\lt<w, \partial^n_x\rt>^2-4F_{x_0}\lt<w, \partial^n_x\rt>^2.
\ee
From \eqref{nc32} we get that
\be\label{nc35}
\begin{aligned}
&F^{ij}_{x_0}\lt(\frac{\partial^2\hZ}{\partial x_i\partial x_j}+2\lt<\partial^i_x, \partial^j_y\rt>\frac{\partial^2\hZ}{\partial x_i\partial y_j}\rt.
\lt.+\frac{\partial^2\hZ}{\partial y_i\partial y_j}\rt)\\
&=F^{nn}_{x_0}(4\delta h^x_{nn}\lt<w, \partial^n_x\rt>^2-4F_{x_0}\lt<w, \partial^n_x\rt>^2)-\delta F_{x_0}+F_{x_0}F_{y_0}d\lt<w, \nu_{y_0}\rt>\\
&-\delta F_{x_0}+\delta F_{y_0}\lt<\nu_{x_0}, \nu_{y_0}\rt>+2F_{x_0}\sum\limits_iF^{ii}_{x_0}
-2d\lt<\partial^n_x, \partial^n_y\rt>F^{nn}_{x_0}\nabla_nF_{x_0}\lt<w, \partial^n_y\rt>\\
&-2F_{x_0}F^{ij}_{x_0}\lt(\sum\limits_{i=1}^n\delta_{ij}-\delta_{in}+\lt<\partial^i_x, \partial^n_y\rt>^2\rt)
+2\delta(F_{x_0}-F^{nn}_{x_0}h^x_{nn})\\
&+2\delta F^{nn}_{x_0}h^{x_0}_{nn}\lt<\partial^n_x, \partial^n_y\rt>^2-F_{x_0}F^{ij}_{x_0}h^y_{ij}d\lt<w, \nu_{y_0}\rt>
-\delta F^{ij}_{x_0}h^y_{ij}\lt<\nu_x, \nu_y\rt>\\
&-2\lambda(1-\alpha)^2\sum\limits_{i=1}^{n-1}F^{ii}_{x_0}-2\lambda\lt(1-\alpha^2-2\alpha\lt<\txn, \tyn\rt>^2\rt)F^{nn}_{x_0}\\
&=-4F^{nn}_{x_0}(F_{x_0}-\delta h^x_{nn})\lt<w, \partial^n_x\rt>^2+F_{x_0}d\cos\theta(F_{y_0}-F^{ij}_{x_0}h^y_{ij})
-\cos2\theta\delta(F_{y_0}-F^{ij}_{x_0}h^y_{ij})\\
&-\lt<\partial^n_x, \partial^n_y\rt>F^{nn}_{x_0}\lt<w, \partial^n_y\rt>(4\delta h^x_{nn}\lt<w, \partial^n_x\rt>-4F_{x_0}\lt<w, \partial^n_x\rt>)
+2F_{x_0}F^{nn}_{x_0}\lt(1-\lt<\partial^n_x, \partial^n_y\rt>^2\rt)\\
&-2\delta F^{nn}_{x_0}h^x_{nn}\lt(1-\lt<\partial^n_x, \partial^n_y\rt>^2\rt)-2\lambda(1-\alpha)^2\sum\limits_{i=1}^{n-1}F^{ii}_{x_0}\\
&-2\lambda\lt(1-\alpha^2-2\alpha\lt<\txn, \tyn\rt>^2\rt)F^{nn}_{x_0}\\
&=F^{nn}_{x_0}(F_{x_0}-\delta h^x_{nn})
\lt\{-4\lt<w, \partial^n_x\rt>^2+4\lt<w, \partial^n_x\rt>\lt<w, \partial^n_y\rt>\lt<\partial^n_x, \partial^n_y\rt>+2\lt(1-\lt<\partial^n_x,\rt. \lt.\partial^n_y\rt>^2\rt)\rt\}\\
&+(F_{y_0}-F^{ij}_{x_0}h^y_{ij})\lt(2\delta\cos^2\theta-\delta\cos2\theta\rt)-2\lambda(1-\alpha)^2\sum\limits_{i=1}^{n-1}F^{ii}_{x_0}\\
&-2\lambda\lt(1-\alpha^2-2\alpha\lt<\txn, \tyn\rt>^2\rt)F^{nn}_{x_0}\leq0,\\
\end{aligned}
\ee
One can see that the last inequality achieves equality if and only if $F_{y_0}=F^{ij}_{x_0}h^{y}_{ij}$, $\alpha=1$, and $\lt<\txn, \tyn\rt>^2=1.$
Therefore, we can always perturb $\delta$ such that the last inequality is strictly less than $0$, which leads to a contradiction.
\end{proof}

\medskip
\subsection{$\alpha$-Andrews flow in viscosity sense}
Before we state our theorem, we need the following definitions, which generalizes definitions in \cite{HK} for mean curvature flow.
\begin{defin}\label{denc1}(Andrews condition)
If $\Omega\subset\R^n$ is a smooth, closed admissible domain and $\alpha>0,$ then we say $\partial\Omega$ satisfies the
$\alpha$-Andrews condition if for every $P\in\partial\Omega$ there are closed balls $\bar{B}_{int}\subseteq\Omega$
and $\bar{B}_{ext}\subseteq\R^n\setminus\text{Int}(\Omega)$ of radius at least $\frac{\alpha}{F_P}$
that are tangent to $\partial\Omega$ at $P$ from the interior and exterior of $\Omega,$ respectively. A smooth curvature flow
$\{\Gamma_t\subseteq\R^n\}_{t\in I}$ is $\alpha$-Andrews if every time slice satisfies the $\alpha$-Andrews condition.
\end{defin}

\begin{defin}\label{denc2}(Viscosity general curvature)
Let $\Omega\subseteq \R^n$ be a closed set. If $P\in\partial\Omega,$ then the viscosity general curvature of $\Omega$
at $P$ is
\be\label{nc36}
F(P)=\inf\{F_{\partial X}(P)|X\subseteq\Omega\,\, \mbox{is a compact smooth domain,}\,\, P\in\partial X \},
\ee
where $F_{\partial X}(P)$ denotes the general curvature of $\partial X$ at $P$ with respect to the inward pointing normal.
The infimum of the empty set is $\infty.$
\end{defin}

\begin{defin}\label{denc3}(Viscosity $\alpha$-Andrews condition)
A closed set $\partial\Omega$ satisfies the viscosity $\alpha$-Andrews condition if $F(P)\in[0, \infty]$
at every boundary point, and there are interior and exterior balls $\bar{B}_{int},$ $\bar{B}_{ext}$ with radius at least $\alpha/F(P)$
passing through $P.$ A level set flow $\{\Gamma_t\}$ is $\alpha$-Andrews if every time slice satisfies the viscosity $\alpha$-Andrews condition.
\end{defin}

Using these definitions, Theorem \ref{thws1}, and Theorem \ref{thci0} we can extend Andrews' Theorem to the nonsmooth setting:
\begin{theorem}\label{thnc2}
Let $\Gamma_0$ be a hypersurface satisfies the following conditions: $\Gamma_0$ can be approached by a sequence of smooth hypersurfaces, which have positive general curvature and satisfy $\alpha$-Andrews condition. If $\{\Gamma_t\}$ is a compact level set general curvature flow with initial hypersurface $\Gamma_0,$ then $\{\Gamma_t\}$ is an $\alpha$-Andrews level set flow.
\end{theorem}

\section*{Acknowledgement}
The author would like to thank Professor Joel Spruck for helpful discussions about this work.

\end{document}